\documentclass[11pt]{amsart}

\usepackage[dvipsnames, table]{xcolor}
\usepackage{amsthm}
\theoremstyle{plain}
\usepackage{amssymb}
\usepackage{marvosym}
\usepackage{bm}
\usepackage{mathrsfs}
\usepackage{enumerate}  
\usepackage{mathtools}
\usepackage{tikz-cd}
\usepackage{graphicx}
\usepackage{float}
\usepackage[titletoc,title]{appendix}
\usepackage{marginnote}
\usepackage[disable]{todonotes}
\usepackage{etoolbox}
\usepackage[all]{xy}
\makeatletter
\patchcmd{\Ginclude@eps}{"#1"}{#1}{}{}
\makeatother
\usepackage[outdir=./]{epstopdf}
\usepackage[numbers]{natbib}
\usepackage[utf8]{inputenc}
\usepackage[english]{babel}
\usepackage{changepage}
\usepackage{makecell}
\usepackage{enumitem}

\definecolor{lightblue}{HTML}{1F88CD}
\definecolor{lightgrey}{HTML}{727272}
\definecolor{lightblue2}{HTML}{009EC1}
\definecolor{mypink}{HTML}{FD00B0}
\definecolor{lightred}{HTML}{ff4d4d}

\usepackage{hyperref}
\hypersetup{
	colorlinks=true,
    linkcolor={OliveGreen},
    citecolor={blue},
	urlcolor={black}
}

\newtheorem*{theorem*}{Theorem}
\newtheorem{theorem}{Theorem}[section]
\newtheorem{corollary}[theorem]{Corollary}
\newtheorem{lemma}[theorem]{Lemma}
\newtheorem{conjecture}[theorem]{Conjecture}

\newtheorem{proposition}[theorem]{Proposition}
\theoremstyle{definition}

\theoremstyle{definition}
\newtheorem{definition}[theorem]{Definition}
\theoremstyle{definition}
\newtheorem{remark}[theorem]{Remark}
\theoremstyle{definition}

\theoremstyle{definition}

\theoremstyle{definition}

\theoremstyle{definition}

\theoremstyle{definition}

\theoremstyle{definition}
\newtheorem{question!}[theorem]{Question!}
\theoremstyle{definition}

\makeatletter
\newcommand*\sbt{\mathpalette\sbt@{.75}}
\newcommand*\sbt@[2]{\mathbin{\vcenter{\hbox{\scalebox{#2}{$\m@th#1\bullet$}}}}}
\makeatother

\newcommand{\ra}{\rightarrow}
\newcommand{\GL}{\widetilde{\mathrm{GL}}^+(2,\mathbb{R})}
\newcommand{\xra}{\xrightarrow}

\newcommand{\sst}{\subset}

\newcommand{\bL}{\bm{\mathrm{L}}}

\newcommand{\D}{\mathrm{D}}

\newcommand{\ZZ}{\mathbb{Z}}

\newcommand{\CC}{\mathbb{C}}
\newcommand{\PP}{\mathbb{P}}

\newcommand{\K}{\mathrm{K}}
\newcommand{\gldim}{\mathrm{gldim}}
\newcommand{\homdim}{\mathrm{homdim}}
\newcommand{\Knum}{\mathrm{K}_{\mathrm{num}}(\mathcal{K}u(X))}

\renewcommand{\Re}{\operatorname{Re}}
\renewcommand{\Im}{\operatorname{Im}}

\DeclareMathOperator{\Aut}{Aut}

\DeclareMathOperator{\identity}{id}

\DeclareMathOperator{\rk}{rk}

\DeclareMathOperator{\Ext}{Ext}

\DeclareMathOperator{\Hom}{Hom}
\DeclareMathOperator{\RHom}{RHom}

\DeclareMathOperator{\ext}{ext}

\DeclareMathOperator{\Pic}{Pic}

\DeclareMathOperator{\cone}{cone}
\DeclareMathOperator{\Stab}{Stab}
\DeclareMathOperator{\Gr}{Gr}

\newcommand{\cA}{\mathcal{A}}
\newcommand{\cE}{\mathcal{E}}

\newcommand{\cH}{\mathcal{H}}

\newcommand{\cI}{\mathcal{I}}
\newcommand{\cJ}{\mathcal{J}}
\newcommand{\cT}{\mathcal{D}}

\newcommand{\Ku}{\mathcal{K}u}
\newcommand{\cP}{\mathcal{P}}

\DeclareMathOperator{\oh}{\mathcal{O}}

\usetikzlibrary{decorations.pathmorphing}

\title[Stability manifolds of Kuznetsov components of prime Fano threefolds]{Stability manifolds of Kuznetsov components of prime Fano threefolds} 


\author{Changping Fan}
\address{College of Mathematics, Sichuan University, Chengdu, Sichuan Province 610064, P. R. China}
\email{leslie.fan1390@gmail.com}

\author{Zhiyu Liu}
\address{Institute for Advanced Study in Mathematics, Zhejiang University, Hangzhou, Zhejiang Province 310030, P. R. China}
\email{jasonlzy0617@gmail.com}

\author{Songtao Ma}
\address{Department of Mathematics, University of Michigan - Ann Arbor, Ann Arbor, Michigan 48104, United States}
\email{skenma@umich.edu}

\date{\today}

\begin{document}

\begin{abstract}
Let $X$ be a cubic threefold, quartic double solid or Gushel--Mukai threefold, and $\mathcal{K}u(X)\subset \mathrm{D}^b(X)$ be its Kuznetsov component. We show that a stability condition $\sigma$ on $\mathcal{K}u(X)$ is Serre-invariant if and only if its homological dimension is at most $2$.  As a corollary, we prove that all Serre-invariant stability conditions on $\mathcal{K}u(X)$ form a contractible connected component of the stability manifold.
\end{abstract}

\maketitle

{
\hypersetup{linkcolor=blue}
\setcounter{tocdepth}{1}
\tableofcontents
}

\section{Introduction}

Motivated by the concept of $\Pi$-stability condition on string theory by Douglas, the notion of a stability condition on a triangulated category $\cT$ was introduced by Bridgeland in \cite{bridgeland}. It is proved in \cite{bridgeland} that the set $\Stab(\cT)$ of all stability conditions (with respect to a fixed lattice) on $\cT$ has a natural structure of the complex manifold. In recent years, the geometry of $\Stab(\cT)$ has played a central role in the theory of homological mirror symmetry, representation
theory, symplectic geometry, and moduli of sheaves, see e.g.~\cite{haiden:flat-surface,bridgeland:quadratic-diff,bayer:derived-auto-K3,bayer:projectivity-k3,bayer:mmp-k3,halpern-leistern:NMMP,bridgeland:HPK-by-DT,bridgeland:geo-from-DT}.

One of the most fundamental examples is $\cT=\D^b(C)$ where $C$ is a smooth projective curve. Due to \cite{okada:stab-P1} and \cite{macri2007stability}, $\Stab(\D^b(C))$ is a contractible manifold and has a complete description. More general, there is
a folklore conjecture that if $\Stab(\cT)$ is non-empty, then it is simply connected. However, with current techniques, it still seems to be impossible to prove some basic properties of $\Stab(\cT)$, even for connectedness. Therefore, people found that it was better to study a distinguished connected component of $\Stab(\cT)$. For example, when $\cT=\D^b(X)$ for a smooth projective variety $X$, we say a component is distinguished if it contains a geometric stability condition, with respect to which skyscraper sheaves of closed points are all stable with the same phase. Results along this direction include \cite{bridgeland:stab-A2-quiver,bayer2016space,bayer:derived-auto-K3,fu:stab,li:stab-P2,qiu:stab-dynkin,qiu:stab-and-braid-group,dimitrov2019bridgeland,hannah:stab-free-abel-quotient,rekusi:contract-stab-surface}, etc.

The examples above mainly focus on the situation when $\cT$ is the bounded derived category of a certain smooth variety or the derived category of representations of a well-described quiver or algebra. On the other hand, by a series of work \cite{kuznetsov2003derived,kuznetsov2005derived,kuznetsov2009derived}, we can associate a very interesting triangulated subcategory $\Ku(X)$ of $\D^b(X)$ to a prime Fano threefold $X$, called \emph{the Kuznetsov component} of $X$ (cf.~Section \ref{section-Ku}), which are not of this form in many cases. They come from semiorthogonal decompositions and have been widely studied in recent times.  For instance, when $X\subset \PP^4$ is a smooth cubic threefold, $\Ku(X)$ is defined as the right orthogonal of line bundles $\oh_X$ and $\oh_X(1)$, which is a $\frac{5}{3}$-Calabi--Yau category in the sense of \cite{kuz:fractional-CY}, and is not equivalent to $\D^b(X)$ or the derived category of a nice quiver or algebra. And by a celebrated result of Orlov \cite{orlov:matrix-factor-equiv}, $\Ku(X)$ is also equivalent to the category of graded matrix factorization of $X$, which plays an important role in the theory of Landau--Ginzburg models. 

Moreover, for many prime Fano threefolds $X$, the category $\Ku(X)$ can be regarded as a non-commutative smooth projective curve for many reasons: it is a non-commutative smooth projective variety in the sense of \cite{orlov:noncomm-scheme}; the rank of its numerical Grothendieck group is the same as $\D^b(C)$; its Serre dimension is smaller than $2$, etc. Therefore, it is natural to study $\Stab(\Ku(X))$ and expect its geometry is similar to $\Stab(\D^b(C))$,

All prime Fano threefolds are classified in \cite{iskovskikh1999fano} by their index and degree $d$. When $X$ is of index $\geq 3$, or index $2$ and degree $d> 3$, or index $1$ and degree $d> 10$, $\Ku(X)$ is either equivalent to $\D^b(C)$ for a smooth projective curve $C$ or the derived category of representations of a generalized Kronecker quiver, whose $\Stab(\Ku(X))$ is contractible and has been studied in \cite{macri2007stability,dimitrov2019bridgeland}. When $X$ is of index $2$ and degree $d<2$ or index $1$ and degree $d<10$, $\Ku(X)$ behaves badly and it is not reasonable to regard $\Ku(X)$ as a non-commutative curve in these cases. Therefore, the only remaining examples are cubic threefolds (index $2$ and degree $d=3$), quartic double solids (index $2$ and degree $d=2$), and Gushel--Mukai threefolds (index $1$ and degree $d=10$). In these cases, skyscraper sheaves are not in $\Ku(X)$, so there is no notion of geometric stability conditions. However, we can consider the projection objects of skyscrapers into $\Ku(X)$, and these objects are stable with respect to \emph{Serre-invariant} stability conditions on $\Ku(X)$ (cf.~\cite{JLZ2022,liu:cat-torelli-delpezzo}). Then we say a connected component of $\Stab(\Ku(X))$ is distinguished if it contains a Serre-invariant stability condition (see Section \ref{sec-s-inv} for the definition).

By Remark \ref{stab-rmk}, $\Ku(X)$ contains a unique distinguished component. Our first theorem states that in these three cases, the distinguished connected component of $\Stab(\Ku(X))$ actually consists of Serre-invariant stability conditions and is contractible.

\begin{theorem}[{Corollary \ref{cor-component}}]\label{thm-1.1}
Let $X$ be a cubic threefold, quartic double solid or Gushel--Mukai threefold. Then the subspace of all Serre-invariant stability conditions is a contractible connected component of $\Stab(\Ku(X))$.
\end{theorem}

When $X$ is a quartic double solid, Theorem \ref{thm-1.1} follows from \cite[Theorem 1.3]{pertusi2020some} and Theorem \ref{unqiue}, in which a geometric description of the Serre functor of $\Ku(X)$ is used and can only work for quartic double solids. Our method is different from \cite[Theorem 1.3]{pertusi2020some} and can be generalized to other cases.

In our paper, we prove the following criterion of Serre-invariant stability conditions:

\begin{theorem}[{Theorem \ref{cri_thm}}]\label{thm-1.2}
Let $X$ be a cubic threefold, quartic double solid or Gushel--Mukai threefold, and $\sigma$ be a stability condition on $\Ku(X)$. Then the following conditions are equivalent:

\begin{enumerate}
    \item $\homdim(\sigma)\leq 2$,

    \item $\gldim(\sigma)\leq 2$, and

    \item $\sigma$ is Serre-invariant.
\end{enumerate}
\end{theorem}

Here, $\homdim(\sigma)$ is the \emph{homological dimension} of $\sigma$ and $\gldim(\sigma)$ is the \emph{global dimension} of $\sigma$ (cf.~Section \ref{subsec-gldim}).

The implications (3) $\Rightarrow$ (2) and (2) $\Rightarrow$ (1) in Theorem \ref{thm-1.2} are easy. However, it is quite surprising for the authors that (1) implies (3), since (2) and (3) are preserved by $\GL$-action, while the condition (1) is not in general.

Finally, we conjecture that Theorem \ref{thm-1.2} holds for any stability condition on $\Ku(X)$ (see Conjecture \ref{conj}), which together with Theorem \ref{thm-1.1} implies $\Stab(\Ku(X))\cong \GL\cong \CC\times \mathbb{H}$.

\subsection*{Notation and conventions} \leavevmode
\begin{itemize}
    

    \item We denote the phase and slope with respect to a stability condition $\sigma$ by $\phi_{\sigma}$ and $\mu_{\sigma}$,  respectively. The maximal/minimal phase of the Harder--Narasimhan factors of a given object will be denoted by $\phi_{\sigma}^+$ and $\phi_{\sigma}^-$,  respectively.

    \item We use $\hom$ and $\ext^{i}$ to represent the dimension of the vector spaces $\Hom$ and~$\Ext^{i}$. We denote $\RHom(-,-)=\bigoplus_{i\in \ZZ} \Hom(-,-[i])[-i]$.

    \item All triangulated categories are assumed to be $\CC$-linear of finite type, i.e.~for any two objects $E,F$ we have $\sum_{i\in \ZZ} \ext^i(E,F)<+\infty$.
    
    \item We denote the numerical class in the numerical Grothendieck group by $[E]$ for any object $E$. In our setting, giving a numerical class is equivalent to giving a Chern character.

\end{itemize}

\subsection{Plan of the paper}

In Section \ref{section-Ku}, we introduce Kuznetsov components $\Ku(X)$ and recollect some basic properties. In Section \ref{section-stability}, we recall some basic definitions and properties of Bridgeland stability conditions on $\Ku(X)$. Then we review the definition of global dimension and homological dimension, and prove some general preliminary results in Section \ref{section-pre}. Finally, we prove our main results Theorem \ref{cri_thm_1}, Theorem \ref{cri_thm} and Corollary \ref{cor-component} in Section \ref{section-main}.

\subsection*{Acknowledgements}
We would like to thank Chunyi Li and Alexander Perry for their careful reading and many useful suggestions. The third author would also like to thank Dima Arinkin and Andrei Caldararu for useful discussions. Part of the work was finished when the second author visited the Hausdorff Research Institute for Mathematics (HIM). He is grateful for the wonderful working environment and hospitality.

\section{Kuznetsov components} \label{section-Ku}


In this paper, we consider a smooth projective threefold $X$ with $-K_X$ ample and Picard rank one, i.e.~a prime Fano threefold. In the following, we are mainly concerned about three types of prime Fano threefolds:

\begin{itemize}
    \item cubic threefolds: $X\subset \PP^4$ is a smooth cubic hypersurface,

    \item quartic double solids: there is a double cover $X\to \PP^3$ branched along a smooth quartic surface, and

    \item Gushel--Mukai threefolds: a smooth intersection $$X=\mathrm{Cone}(\Gr(2,5))\cap Q,$$ 
    where $\mathrm{Cone}(\Gr(2,5))\subset \PP^{10}$ is the projective cone over the \text{Plücker} embedded Grassmannian $\Gr(2,5)\subset \PP^9$, and $Q\subset \PP^{10}$ is a quadric hypersurface in a linear subspace $\PP^{7}\subset \PP^{10}$.
\end{itemize}

When $X$ is a cubic threefold or quartic double solid, we have a semiorthogonal decomposition
\[\D^b(X)=\langle \Ku(X),\oh_X,\oh_X(1)\rangle,\]
where $\oh_X(1)$ is the ample generator of $\Pic(X)$. 

When $X$ is a Gushel--Mukai threefold, there is also a semiorthogonal decomposition given in \cite{kuznetsov2018derived}:
\[\D^b(X)=\langle \Ku(X), \oh_X, \cE^{\vee}_X\rangle,\]
where $\cE_X$ is the pullback of the tautological subbundle on $\Gr(2,5)$ along $X\to \Gr(2,5)$.

\begin{definition} \label{kuznetsov components definition}
Let $X$ be a cubic threefold, quartic double solid, or Gushel--Mukai threefold.  The \emph{Kuznetsov component} of $X$ is the admissible triangulated subcategory $\Ku(X)\subset \D^b(X)$ constructed above. 
\end{definition}

Let $\K(\Ku(X))$ denote the Grothendieck group of $\Ku(X)$. We have the bilinear Euler form on $\K(\Ku(X))$ defined by
\[ \chi([E], [F]) = \sum_{i \in \ZZ} (-1)^i \ext^i(E, F) \]
for $[E], [F] \in \K(\Ku(X))$.
The \emph{numerical Grothendieck group} of $\Ku(X)$ is defined as $$\Knum := \K(\Ku(X)) / \ker \chi.$$ Therefore, the Euler form on $\K(\Ku(X))$ can be defined on $\Knum$ as well, which we also denote by $\chi(-,-)$.

The Serre functor of a Kuznetsov component can be computed from the Serre functor of $X$ and mutation functors. We define
\[\mathbf{O}:=\bL_{\oh_X}\circ (-\otimes \oh_X(1)),\]
where $\bL_{\oh_X}$ is the left mutation functor along $\oh_X$. When $X$ is a cubic threefold or a quartic double solid, it is an autoequivalence of $\Ku(X)$.

\begin{itemize}
    \item When $X$ is a cubic threefold, we have $S_{\Ku(X)}\cong \mathbf{O}[1]$. In this case, $S^3_{\Ku(X)}\cong [5]$. The functor $S_{\Ku(X)}$ acts non-trivially on $\Knum$.

    \item When $X$ is a quartic double solid or a Gushel--Mukai threefold, we have $S^2_{\Ku(X)}\cong [4]$. In this case we define $\tau:=S_{\Ku(X)}[-2]$. The functor $S_{\Ku(X)}$ (and hence $\tau$) acts trivially on $\Knum$.
\end{itemize}

Finally, we give a description of $\Knum$ and its Euler form.

When $X$ is a cubic threefold or a quartic double solid, a computation using \cite[pp.6]{kuznetsov2009derived} shows that, the rank two lattice $\Knum$ is generated by 
\[v=[\cI_l],~\text{and}~w=[\mathbf{O}(\cI_l)[1]],\]
where $\cI_l$ is the ideal sheaf of a line $l\subset X$. The Euler form $\chi(-,-)$ with respect to the basis $v$ and $w$ is given by the matrix
\begin{equation} \label{euler_form_cubic}
   \begin{pmatrix}  
-1 & -1 \\
0 & -1
\end{pmatrix}  
\end{equation}
when $X$ is a cubic threefold, and 
\begin{equation} \label{euler_form_quartic}
   \begin{pmatrix}  
-1 & 0 \\
0 & -1
\end{pmatrix}  
\end{equation}
when $X$ is a quartic double solid.

For a Gushel--Mukai threefold $X$, $\Knum$ is computed in \cite{kuznetsov2018derived} and \cite[pp.5]{kuznetsov2009derived}. The numerical Grothendieck group $\mathrm{K}_{\mathrm{num}}(\cA_X)$ is a rank 2 lattice with basis vectors
\[v=[\cI_C],~\text{and}~w=[F],\]
where $C\subset X$ is a conic and $F$ is a rank two slope-stable sheaf with $c_1(F)=-1, c_2(F)=5$ and $c_3(F)=0$. The Euler form  $\chi(-,-)$ with respect to the basis is 
\begin{equation} \label{euler_form_GM}
   \begin{pmatrix}  
-1 & 0 \\
0 & -1
\end{pmatrix} . 
\end{equation}

\section{Bridgeland stability conditions} \label{section-stability}

In this section, we recall basic definitions of Bridgeland stability conditions on triangulated categories. Then we focus on the stability conditions on Kuznetsov components. We follow \cite{bridgeland} and  \cite{bayer2017stability}.

\subsection{Bridgeland stability conditions}

Let $\cT$ be a triangulated category and $\K(\cT)$ its Grothendieck group. 

\begin{definition}
The \emph{heart of a bounded t-structure} on $\cT$ is an abelian subcategory $\cA \sst \cT$ such that the following conditions are satisfied:
\begin{enumerate}
    \item for any $E, F \in \cA$ and $n <0$, we have $\Hom(E, F[n])=0$;
    \item for any object $E \in \cT$ there exist objects $E_i \in \cA$ and maps
    \[ 0=E_0 \xrightarrow{\pi_1} E_1 \xrightarrow{\pi_2} \cdots \xra{\pi_m} E_m=E \]
    such that $\cone(\pi_i) = A_i[k_i]$ where $A_i \in \cA$ and the $k_i$ are integers such that $k_1 > k_2 > \cdots > k_m$.
\end{enumerate}
\end{definition}


Fix a surjective morphism $v \colon \K(\cT) \ra \Lambda$ to a finite rank lattice $\Lambda$.

\begin{definition}
A \emph{stability condition} on $\cT$ is a pair $\sigma = (\cA, Z)$ where $\cA$ is the heart of a bounded t-structure on $\cT$, and $Z : \Lambda \ra \CC$ is a group homomorphism such that 
\begin{enumerate}
    \item the composition $Z \circ v : \K(\cA) = \K(\cT) \ra \CC$ satisfies: for any $E\neq 0 \in \cT$ we have $\Im Z(v(E)) \geq 0$ and if $\Im Z(v(E)) = 0$ then $\Re Z(v(E)) < 0$. From now on, we write $Z(E)$ rather than $Z(v(E))$.
\end{enumerate}
We can define a \emph{slope} $\mu_\sigma$ for $\sigma$ using $Z$. For any $E \in \cA$, set
\[
\mu_\sigma(E) := \begin{cases}  - \frac{\Re Z(E)}{\Im Z(E)}, & \Im Z(E) > 0 \\
+ \infty , & \text{else}.
\end{cases}
\]
We say an object $0 \neq E \in \cA$ is $\sigma$-(semi)stable if $\mu_\sigma(F) < \mu_\sigma(E)$ (respectively $\mu_\sigma(F) \leq \mu_\sigma(E)$) for all proper subobjects $F \sst E$. 
\begin{enumerate}[resume]
    \item Any object $E \in \cA$ has a Harder--Narasimhan (HN) filtration in terms of $\sigma$-semistability;
    \item There exists a quadratic form $Q$ on $\Lambda \otimes \mathbb{R}$ such that $Q|_{\ker Z}$ is negative definite, and $Q(E) \geq 0$ for all $\sigma$-semistable objects $E \in \cA$.
\end{enumerate}
\end{definition}

\begin{definition}
The \emph{phase} of a $\sigma$-semistable object $E\in \cA$ is
\[\phi(E) := \frac{1}{\pi} \mathrm{arg}(Z(E)) \in (0,1].\]
Specially, if $Z(E) = 0$ then $\phi(E) = 1$. If $F = E[n]$, then we define 
\[\phi(F) := \phi(E) + n\]

A \emph{slicing} $\cP$ of $\cT$ consists of full additive subcategories $\cP(\phi) \subset \cT$ for each $\phi \in \mathbb{R}$ satisfying
\begin{enumerate}[resume]
\item for $\phi\in (0,1]$, the subcategory $\cP(\phi)$ is given by the zero object and all $\sigma$-semistable objects whose phase is $\phi$;
\item for $\phi + n$ with $\phi\in (0,1]$ and $n\in \mathbb{Z}$, we set $\cP(\phi + n) := \cP(\phi)[n]$.
\end{enumerate}
\end{definition}

We will use both notations $\sigma = (\cA,Z)$ and $\sigma = (\cP,Z)$ for a stability condition $\sigma$ with heart $\cA = \cP((0,1])$ where $\cP$ is the slicing of $\sigma$.

In this paper, we let $\cT$ be the Kuznetsov component and $\Lambda$ be the numerical Grothendieck group $\Knum$, which is $\K(\Ku(X))$ modulo the kernel of the Euler form $\chi(-, -)$. 

For a stability condition $\sigma=(\cA, Z)$ on $\cT$, we define the \emph{homological dimension} of $\cA$ as the smallest integer $\homdim(\cA)$ such that $\Hom(A,B[n])=0$ for any $n>\homdim(\cA)$. The homological dimension of a stability condition is defined as the homological dimension of its heart.

\subsection{The stability manifold}

The set of stability conditions $\Stab(\cT)$ has a natural topology induced by a generalized metric. Moreover, the following theorem of Bridgeland states that the generalized metric space $\Stab(\cT)$ is a complex manifold.

\begin{theorem}
\emph{(Bridgeland Deformation Theorem, \cite{bridgeland})} The continuous map $$\mathcal{Z}: \Stab(\cT) \rightarrow \Hom(\Lambda,\mathbb{C}),\quad (\cA,Z)\mapsto Z$$  is a local homeomorphism. In particular, the generalized metric space $\Stab(\cT)$ has the structure of a complex manifold of dimension $\rk(\Lambda)$.
\end{theorem}

Next, we recall two natural group actions on 
$\Stab(\cT)$.
\begin{enumerate}
    \item An element $\tilde{g} = (g, G)$ in the universal covering $\widetilde{\mathrm{GL}}^+(2,\mathbb{R})$ of the group $\mathrm{GL}^+(2,\mathbb{R})$ consists of an increasing function $g: \mathbb{R} \rightarrow \mathbb{R}$ such that $g(\phi+1) = g(\phi) + 1$ and matrix $G\in \mathrm{GL}^+(2,\mathbb{R})$ with $\det(G)>0$. It acts on the right on the stability manifold by $\sigma \cdot \tilde{g} := (G^{-1}\circ Z,\cP(g(\phi))$ for any $\sigma = (\cP,Z)\in \Stab(\cT)$ (see \cite[Lemma 8.2]{bridgeland}).
    
    \item Let $\Aut_\Lambda(\cT)$ be the group of exact autoequivalences of $\cT$, whose action $\Phi_*$ on $\K(\cT)$ is compatible with $v\colon \K(\cT)\to \Lambda$. For $\Phi\in \Aut_\Lambda(\cT)$ and $\sigma=(\cP, Z) \in \Stab(\cT)$, we define a left action of the group of linear exact autoequivalences $\Aut_\Lambda(\cT)$ by $\Phi\cdot \sigma = (\Phi(\cP), Z\circ \Phi_*^{-1})$, where $\Phi_*$ is the automorphism of $\K(\cT)$ induced by $\Phi$.
\end{enumerate}

\subsection{Serre-invariant stability conditions on Kuznetsov components} \label{sec-s-inv}

\begin{theorem}[{\cite[Proposition 6.8]{bayer2017stability}}] \label{stab_Ku}
Let $X$ be a cubic threefold, quartic double solid or Gushel--Mukai threefold. Then there is a family of stability conditions on $\Ku(X)$ with respect to the rank two lattice $\Knum$ and natural surjection $\K(\Ku(X))\twoheadrightarrow \Knum$.
\end{theorem}

\begin{definition}\label{def-S-invariant}
Let $\sigma$ be a stability condition on a triangulated category $\cT$ with the Serre functor $S_{\cT}$. It is called \emph{Serre-invariant} if $S_{\cT} \cdot \sigma=\sigma \cdot \tilde{g}$ for some $\tilde{g}\in\widetilde{\mathrm{GL}}^+(2,\mathbb{R})$.
\end{definition}

By virtue of \cite[Corollary 5.5]{pertusi2020some} and \cite[Theorem 1.1]{Pertusi2021serreinv}, all stability conditions constructed in Theorem \ref{stab_Ku} are Serre-invariant.

We recall several properties of Serre-invariant stability conditions on Kuznetsov components.

\begin{proposition}\label{S_prop}
Let $X$ be a cubic threefold, quartic double solid or Gushel-Mukai threefold and $\sigma$ be a Serre-invariant stability condition on $\Ku(X)$.
\begin{enumerate}
    \item the homological dimension of the heart of $\sigma$ is two,

    \item For any $E$ and $F$ in $\Ku(X)$ with phases $\phi^+_{\sigma}(E)< \phi^-_{\sigma}(F)$, we have $\mathrm{Hom}(E, F[2])=0$.
    
    \item If $X$ is a cubic threefold and $E\in \Ku(X)$ is $\sigma$-semistable, then $\Ext^2(E, E)=0$.

    \item If $X$ is a quartic double solid or Gushel--Mukai threefold and $E\in \Ku(X)$ is $\sigma$-stable, then $\ext^2(E, E)\leq 1$.
    
    \item $\mathrm{ext}^1(E,E)\geq 2$ for every non-zero object $E\in \Ku(X)$.

    \item If $\ext^1(E, E)\leq 3$, then $E$ is $\sigma$-stable.
\end{enumerate}
\end{proposition}

\begin{proof}
See \cite[Section 5]{pertusi2020some}, \cite[Section 4.7]{JLLZ} and \cite[Proposition 3.4]{FeyzbakhshPertusi2021stab}.
\end{proof}

\begin{lemma}\label{small-ext-stable}
Let $X$ be a cubic threefold, quartic double solid or Gushel-Mukai threefold, and $E\in \Ku(X)$ be a non-zero object with $\ext^1(E, E)\leq 3$. Then $E$ is stable with respect to any Serre-invariant stability condition and $[E]\in \Knum$ is primitive.
\end{lemma}

\begin{proof}
The stability of $E$ follows from (6) of Proposition \ref{S_prop}. Thus $\hom(E,E)=1$. Moreover, by (3) and (4) of Proposition \ref{S_prop}, we see $\chi(E, E)=-1$ when $X$ is a cubic threefold, and $\chi(E, E)=-1$ or $-2$ otherwise. Then by a computation using Euler form, we know that $[E]\in \Knum$ is primitive.
\end{proof}

\begin{lemma}\label{bound-phase}

Let $X$ be a cubic threefold, quartic double solid or Gushel--Mukai threefold. Suppose $\sigma$ is a Serre-invariant stability on $\Ku(X)$ and $E$ is a $\sigma$-semistable object.

\begin{enumerate}
    \item When $X$ is a cubic threefold, we have
    \[\phi_{\sigma}(E)+1\leq \phi_{\sigma}(S_{\Ku(X)}(E))<\phi_{\sigma}(E)+2.\]
    The first inequality is strict if $E$ is $\sigma$-stable.

    \item When $X$ is a quartic double solid or Gushel--Mukai threefold, we have
    \[\phi_{\sigma}(S_{\Ku(X)}(E))=\phi_{\sigma}(E)+2.\]
\end{enumerate}

\end{lemma}

\begin{proof}
(1) follows from \cite[Proposition 3.3(a),(d)]{FeyzbakhshPertusi2021stab}. In the case of (2), we also have $\phi_{\sigma}(E)+1\leq \phi_{\sigma}(S_{\Ku(X)}(E))\leq \phi_{\sigma}(E)+2$. Since $[S_{\Ku(X)}(E)]=[E]$, we have
\[\phi_{\sigma}(E)-\phi_{\sigma}(S_{\Ku(X)}(E))\in 2\ZZ.\]
Thus the we get $\phi_{\sigma}(S_{\Ku(X)}(E))=\phi_{\sigma}(E)+2$.
\end{proof}

Finally, we recall the following uniqueness result of Serre-invariant stability conditions.

\begin{theorem}[{\cite[Theorem 4.25]{JLLZ}, \cite[Theorem 3.1]{FeyzbakhshPertusi2021stab}}]
\label{unqiue}

Let $X$ be a cubic threefold, quartic double solid or Gushel--Mukai threefold. Then all Serre-invariant stability conditions on $\Ku(X)$ are in the same $\GL$-orbit.
\end{theorem}

\begin{remark}\label{stab-rmk}
Let $\mathsf{K}$ be the subset of all Serre-invariant stability conditions with induced topology from $\Stab(\Ku(X))$. Then Theorem \ref{unqiue} implies $\mathsf{K}\cong \GL\cong \mathbb{C}\times \mathbb{H}$.
\end{remark}

\section{Preliminary results}\label{section-pre}

In this section, we provide some preliminary results.

We start with a useful lemma. We say a non-zero element $a$ in a lattice $\Lambda$ is \emph{primitive} if it can not be written as $a=nb$ for $b\in \Lambda$ and $n\in \ZZ_{>1}$.

\begin{lemma}\label{des-triangle}
Let $\sigma$ be a stability condition on $\cT$ with respect to the lattice $\Lambda$ and $v\colon K(\cT)\to \Lambda$. If a non-zero object $E\in \cT$ is not $\sigma$-stable and $v(E)\in \Lambda$ is primitive, then we can find an exact triangle
\[A\to E\to B\]
such that $A$ is $\sigma$-semistable, $\phi_{\sigma}(A)\geq \phi^+_{\sigma}(B)$ and $\Hom(A,B)=0$. Moreover, we can assume all Jordan--H\"older factors of $A$ are isomorphic to each other. 
\end{lemma}

\begin{proof}
If $E$ is strictly $\sigma$-semistable, since $v(E)$ is primitive, Jordan--H\"older factors of $E$ can not be isomorphic to each other. Then the existence of exact triangle $A\to E\to B$ follows from the existence of Jordan--H\"older filtration of $A$. 

If $E$ is not $\sigma$-semistable, then by the existence of Harder--Narasimhan filtration, we can find an exact triangle $$A'\to E\to B'$$ such that $A'$ is $\sigma$-semistable and $\phi_{\sigma}(A')>\phi^+_{\sigma}(B')$. In particular, we see  $\Hom(A', B')=0$. If Jordan--H\"older factors of $A'$ are isomorphic to each other, then we take $A:=A'$ and $B:=B'$. If $A'$ has at least two non-isomorphic Jordan--H\"older factors, then as in the previous case, from the existence of Jordan--H\"older filtration of $A'$, we can find an exact triangle $A\to A'\to A''$ such that $A$ is $\sigma$-semistable with $\Hom(A, A'')=0$, and all Jordan--H\"older factors of $A$ are isomorphic to each other. We take $B:=\cone(A\to E)$, where the map $A\to E$ is the composition of $A\to A'$ and $A'\to E$. Hence we have the following commutative diagram
\[\begin{tikzcd}
	A & A & 0 \\
	{A'} & E & {B'} \\
	{A''} & B & {B'}
	\arrow[shift left, no head, from=1-1, to=1-2]
	\arrow[from=1-1, to=2-1]
	\arrow[from=2-1, to=3-1]
	\arrow[from=3-1, to=3-2]
	\arrow[from=3-2, to=3-3]
	\arrow[shift left, no head, from=2-3, to=3-3]
	\arrow[from=2-2, to=3-2]
	\arrow[from=1-2, to=1-3]
	\arrow[from=1-3, to=2-3]
	\arrow[from=1-2, to=2-2]
	\arrow[from=2-1, to=2-2]
	\arrow[from=2-2, to=2-3]
	\arrow[no head, from=2-3, to=3-3]
	\arrow[no head, from=1-1, to=1-2]
\end{tikzcd}\]
with all rows and columns are exact triangles. Since $\phi_{\sigma}(A)=\phi_{\sigma}(A')>\phi_{\sigma}^+(B')$, we obtain $\Hom(A, B')=0$. Hence from $\Hom(A, A'')=0$, we see $\Hom(A,B)=0$.
\end{proof}

We have the following standard spectral sequence, see e.g.~\cite[Lemma 2.27]{pirozhkov2020admissible}.

\begin{lemma}\label{spec-seq}
Let $X$ be a smooth projective variety. Suppose that there are two exact triangles in $\D^b(X)$:
\[A_1\to B_1\to C_1, \quad A_2\to B_2\to C_2.\]
There exist a spectral sequence which degenerates at $E_3$ and converges to $\Ext^*(C_1, C_2)$, with $E_1$-page
\[E^{p,q}_1= \left\{
	\begin{aligned}
	\mathrm{Ext}^q(B_1,A_2) & , ~ p=-1 \\
	\mathrm{Ext}^q(A_1,A_2)\oplus \mathrm{Ext}^q(B_1,B_2) & , ~ p=0 \\
	\mathrm{Ext}^q(A_1, B_2) & , ~ p=1 \\
	0 &, ~  p\notin [-1,1]
	\end{aligned}
	\right.\]
Moreover, differentials $d_1^{p,q}\colon E^{p,q}_1\to E^{p+r, q-r+1}_1$ are given by compositions with morphisms $A_1\to B_1$ and $A_2\to B_2$.
\end{lemma}

The following lemma is a generalization of \cite[Lemma 2.5]{bayer:derived-auto-K3}.

\begin{lemma} \emph{(Weak Mukai Lemma)} \label{mukai lemma}
Let $X$ be a smooth projective variety and 
\[F\to E\to G\]
be an exact triangle in $\D^b(X)$ such that $\mathrm{Hom}(F,G)=\mathrm{Hom}(G,F[2])=0$. Then we have
\[\ext^1(F,F)+\ext^1(G,G)\leq \ext^1(E,E)\]
\end{lemma}

\begin{proof}
Applying the standard spectral sequence in Lemma \ref{spec-seq} to the exact triangle
\[G[-1]\to F\to E,\]
i.e. take $A_1=A_2=G[-1], B_1=B_2=F$ and $C_1=C_2=E$, we have $E_1^{-1,1}=\Hom(F,G)$, $E_1^{0,1}=\Ext^1(G, G)\oplus \Ext^1(F, F)$ and $E_1^{1,1}=\Hom(G,F[2])$ and $E_1^{p,1}=0$ for $p\notin \{-1,0,1\}$. Then by the assumption we get $E_1^{0,1}=E_{\infty}^{0,1}$, which implies the result.
\end{proof}

Specializing in Kuznetsov components, we have the following lemma.

\begin{lemma} \label{sigma_mukai}
Let $X$ be a cubic threefold, quartic double solid or Gushel--Mukai threefold. Assume that there is an exact triangle of $E\in \Ku(X)$
\[F\to E\to G\]
such that $\mathrm{Hom}(F,G)=\mathrm{Hom}(G,F[2])=0$. Then we have 
\[\mathrm{ext}^1(F,F)<\mathrm{ext}^1(E,E),~\text{and}~\mathrm{ext}^1(G,G)<\mathrm{ext}^1(E,E).\]
\end{lemma}

\begin{proof}
This follows from Lemma \ref{mukai lemma} and Proposition \ref{S_prop}.
\end{proof}

The following spectral sequence is constructed in \cite[Proposition 2.4]{okada:stab-cy-surface}.

\begin{proposition}\label{spec-seq-heart}
Let $\cT_0$ be the bounded derived category of an abelian category with enough injective objects, and $\cT\subset \cT_0$ be a full triangulated subcategory of finite type. Then for any heart $\cA$ of $\cT$ and objects $E,F\in \cT$, there exists a spectral sequence with 
\[E_2^{p,q}=\bigoplus_{i\in \ZZ} \Ext^p(\cH^i_{\cA}(E), \cH^{i+q}_{\cA}(F))\]
and converges to $\Ext^*(E,F)$.
\end{proposition}

As a corollary, we see:

\begin{lemma}\label{ext1-heart}
Let $\cT_0$ be the bounded derived category of an abelian category with enough injective objects, and $\cT\subset \cT_0$ be a full triangulated subcategory of finite type. Then for any heart $\cA$ of $\cT$ with $\homdim(\cA)\leq 2$ and object $E\in \cT$, we have
\[\sum_{i\in \ZZ} \ext^1(\cH^i_{\cA}(E), \cH^i_{\cA}(E))\leq \ext^1(E,E).\]
\end{lemma}

\begin{proof}
Since $\homdim(\cA)\leq 2$, we see $E_2^{p,q}=0$ for $p>2, 0>p$ and any $q$. Therefore, we see $E^{1,q}_2=E^{1,q}_{\infty}$. If we take $q=0$, then we have $\dim E^{1,0}_{\infty}=\dim E^{1,0}_2\leq \ext^1(E,E)$, which proves the lemma.
\end{proof}

In the rest of our paper, we always set $\cT_0 = \D^b(\mathrm{QCoh}(X))$ and $\cT=\Ku(X)$.

\subsection{The global dimension function and homological dimension} \label{subsec-gldim}

First, we recall the definition of the homological dimension of a heart and stability condition.

\begin{definition}
Let $\cA$ be the heart of a bounded t-structure of a triangulated category $\cT$. For an integer $n>0$, we say $\cA$ has \emph{homological dimension at most $n$} and denote by $\homdim(\cA)\leq n$ if for any two non-zero objects $E,F\in \cA$, we have
\[\Hom(E,F[k])=0,\forall k>n.\]
We say a stability condition $\sigma$ on $\cT$ has \emph{homological dimension at most $n$} and denote by $\homdim(\sigma)\leq n$ if its heart $\cA$ satisfies $\homdim(\cA)\leq n$.
\end{definition}

There is another notion of dimension for stability conditions, called the global dimension, which is introduced in \cite{IQ:q-stability}.

\begin{definition}
    Let $\sigma$ be a stability condition on a triangulated category $\cT$ with the slicing $\cP$, then the \textit{global dimension} of $\sigma$ is defined by
    \[\gldim(\sigma) := \sup\{\phi_2 - \phi_1|~\Hom(E_1,E_2) \ne 0,~E_i \in \cP(\phi_i),~i = 1~\text{or}~2\}\]
\end{definition}

Therefore, we have a function
\[\gldim\colon \Stab(\cT)\to \mathbb{R}_{\geq 0},\quad \sigma\mapsto \gldim(\sigma),\]
which is continuous by \cite[Lemma 5.7]{IQ:q-stability}.

In the following, we prove some basic properties of global dimension.

\begin{lemma} \label{equiv-cond}
Let $\sigma$ be a stability condition on a triangulated category $\cT$ and $n$ be a positive integer. Then the following are equivalent:

\begin{enumerate}
    \item For any non-zero objects $E$ and $F$ with $\phi_{\sigma}^+(E)<\phi_{\sigma}^-(F)$, we have $\Hom(E,F[n])=0$.

    \item For any $\sigma$-semistable objects $E$ and $F$ with $\phi_{\sigma}(E)<\phi_{\sigma}(F)$, we have $\Hom(E,F[n])=0$.

    \item $\gldim(\sigma)\leq n$.
\end{enumerate}

\end{lemma}

\begin{proof} (1) $\Rightarrow$ (2). It is clear by taking $E$ and $F$ to be semistable.

(2) $\Rightarrow$ (1). Since $\phi_{\sigma}^+(E)<\phi_{\sigma}^-(F)$, for any Harder--Narasimhan factors $E_i$ and $F_j$ of $E$ and $F$, respectively, we have $\phi_\sigma(E_i) < \phi_\sigma(F_j)$, and hence $\Hom(E_i,F_j[n]) = 0$. As there is no morphism between Harder--Narasimhan factors of $E$ and $F[n]$, we also get $\Hom(E, F[n]) = 0$.

(2) $\Rightarrow$ (3). Suppose $\gldim(\sigma) = n + \epsilon$ for some $\epsilon > 0$. This implies there exist $0<\delta \le \epsilon$, together with $\sigma$-semistable objects $E \in \cP(\phi_\sigma(E))$ 
and $F \in \cP(\phi_\sigma(F))$, such that $\phi_\sigma(F) - \phi_\sigma(E) > n+\delta$ and $\Hom(E,F) \ne 0$. But this implies $\Hom(E,(F[-n])[n]) \ne 0$, while $\phi_\sigma(F[-n]) - \phi_\sigma(E) = \phi_\sigma(F) - \phi_\sigma(E) - n > \delta > 0$, which contradicts to our hypothesis. 

(3) $\Rightarrow$ (2). Directly follows from the definition.
\end{proof}

\begin{lemma} \label{global_homo}
Let $\sigma=(\cA, Z)$ be a stability condition on a triangulated category $\cT$ and $\gldim(\sigma)\leq n$ for a positive integer $n$. Then the homological dimension of the heart $\cA$ of $\sigma$ is at most $n$, i.e.~$\Hom(E, F[m])=0$ for any two objects $E, F\in \cA$ and $m\geq n+1$.
\end{lemma}

\begin{proof}
Since $E, F\in \cA$, we know that $\phi^+(E), \phi^-(F)\in (0,1]$. Then the vanishing $\Hom(E, F[m])=\Hom(E[-m+n], F[n])=0$ for $m\geq n+1$ follows from
\[\phi^+(E)-m+n=\phi^+(E[-m+n])\leq \phi^+(E)-1=\phi^+(E[-1]) \leq  0<\phi^-(F)\]
and Lemma \ref{equiv-cond}.
\end{proof}

\begin{lemma}\label{gldim-preserve}
Let $\sigma$ be a stability condition on a triangulated category $\cT$ with the slicing $\cP$ and $\gldim(\sigma)\leq n$ for an integer $n$. Then $\gldim(\sigma\cdot\tilde{g})\leq n$ for any $\tilde{g}\in \GL$. 
\end{lemma}

\begin{proof}
Let $E_1\in \cP(\phi_1)$ and $E_2\in \cP(\phi_2)$ such that $\Hom(E_1, E_2)\neq 0$. Then $n\geq \gldim(\sigma)\geq \phi_2-\phi_1$. Now $\phi_{\sigma\cdot\tilde{g}}(E_i)=g(\phi_{\sigma}(E_i))$, where $g\colon \mathbb{R}\to \mathbb{R}$ be an increasing function and $g(x+n)=g(x)+n$ for any $n\in \ZZ$. Hence we get
\[\phi_{\sigma\cdot\tilde{g}}(E_2)=g(\phi_2)\leq g(\phi_1+n)=g(\phi_1)+n=\phi_{\sigma\cdot\tilde{g}}(E_1)+n.\]
This implies $\phi_{\sigma\cdot\tilde{g}}(E_2)-\phi_{\sigma\cdot\tilde{g}}(E_1)\leq n$, which gives $\gldim(\sigma\cdot\tilde{g})\leq n$.
\end{proof}

\section{A criterion of Serre-invariant stability conditions}\label{section-main}

In this section, we are going to prove the criteria Theorem \ref{cri_thm} and Theorem \ref{cri_thm_1}.

In the following, $X$ will be a cubic threefold, quartic double solid, or Gushel--Mukai threefold. We begin with some lemmas.

\begin{lemma} \label{ext2_inheart}
Let $\sigma=(\cA, Z)$ be a stability condition on $\Ku(X)$ with $\homdim(\cA)\leq 2$. Then for any non-zero object $E\in \Ku(X)$ with $\ext^1(E,E)\leq 3$, $E$ is in $\cA$ up to shift.
\end{lemma}

\begin{proof}
Let $N \in \mathbb{N}$ be the number of non-zero cohomology objects of $E$ with respect to $\cA$. By (5) of Proposition \ref{S_prop} and Lemma \ref{ext1-heart}, we get
\[3\geq \ext^1(E,E)\geq \sum_i \ext^1(\cH^i_{\cA}(E), \cH^i_{\cA}(E))\geq 2N.\]
which implies $N = 1$ and the result follows.

\end{proof}

\begin{lemma}\label{in-same-heart}
Let $\sigma=(\cA,Z)$ and $\sigma'=(\cA', Z')$ be stability conditions on $\Ku(X)$ with homological dimension at most $2$. Let $E_1$ and $E_2$ be two non-zero objects with $[E_1]=[E_2]\in \mathrm{K}_{\mathrm{num}}(\Ku(X))$ such that $E_1[m_1]$ and $E_2[m_2]$ are both in $\cA$. If $E_1[m'_1]$ and $E_2[m'_2]$ are also both in $\cA'$, then $m_1-m_2=m'_1-m'_2$.
\end{lemma}

\begin{proof}
Since $[E_1]=[E_2]$, we know that $m_1-m_2, m_1'-m_2'\in 2\ZZ$. Then we see that 
\[\chi(E_1,E_2)=\chi(E_1[m_1], E_2[m_2])=\chi(E_1[m_1'], E_2[m_2']) < 0.\]
Therefore if $E_1[m_1]$ and $E_2[m_2]$ are both contained in $\cA$, we have 
\[\chi(E_1[m_1], E_2[m_2])=\hom(E_1[m_1], E_2[m_2])-\hom(E_1[m_1], E_2[m_2+1])+\hom(E_1[m_1], E_2[m_2+2])<0.\]
This implies $m_2-m_1+1$ is the unique integer $N$ satisfies
\[\hom(E_1, E_2[N])>\hom(E_1,E_2[n])\]
for any $n\neq N$. On the other hand, the same argument also shows that $m_2'-m_1'+1=N$, so we obtain $m_1-m_2 = m'_1-m'_2$.
\end{proof}


\subsection{Proof of the main theorem}

First, we prove the stability of objects with small $\ext^1$.

\begin{theorem} \label{ext2_stable}
Let $\sigma=(\cA, Z)$ be a stability condition on $\Ku(X)$ with $\homdim(\cA)\leq 2$. Then for any non-zero object $E\in \Ku(X)$ with $\ext^1(E,E)\leq 2$ or $\ext^1(E,E)=3$ and $\ext^2(E,E)=0$, $E$ is $\sigma$-stable.
\end{theorem}

\begin{proof}
By Lemma \ref{small-ext-stable}, $E$ is stable with respect to any Serre-invariant stability condition on $\Ku(X)$ and  $[E]$ is primitive. Moreover, $\RHom(E,E)=\CC\oplus \CC^2[-1]$ when $X$ is a cubic threefold, and $\RHom(E,E)=\CC\oplus \CC^2[-1]$ or $\CC\oplus \CC^3[-1]$ when $X$ is a quartic double solid or Gushel--Mukai threefold.

By virtue of Lemma \ref{ext2_inheart}, we can assume that $E\in \cA$ (up to shift). If $E$ is not $\sigma$-stable, since $[E]$ is primitive, by Lemma \ref{des-triangle} we can find an exact sequence in $\cA$
\begin{equation}\label{des-seq}
    A\to E\to B
\end{equation}
such that $\Hom(A,B)=0$, $\phi_{\sigma}(A)\geq \phi^+_{\sigma}(B)$ and $A$ is $\sigma$-semistable with isomorphic Jordan--H\"older factors. From the exact triangle $B[-1]\to A\to E$ and Lemma \ref{spec-seq}, we have a spectral sequence which degenerates at $E_3$ converging to $\mathrm{Ext}^*(E,E)$ with $E_1$-page being
	\[E^{p,q}_1= \left\{
	\begin{aligned}
	\mathrm{Ext}^q(A,B[-1])=\mathrm{Ext}^{q-1}(A,B) & , ~ p=-1 \\
	\mathrm{Ext}^q(A,A)\oplus \mathrm{Ext}^q(B,B) & , ~ p=0 \\
	\mathrm{Ext}^q(B[-1], A)=\mathrm{Ext}^{q+1}(B,A) & , ~ p=1 \\
	0 &, ~  p\notin [-1,1]
	\end{aligned}
	\right.\]
 Since the homological dimension of $\cA$ is at most $2$, $A,B\in \cA$ and $\Hom(A,B)=0$, we have
 $$ E^{p,q}_1= \begin{array}{cc|cc}
	0 & 0 & 0 & 0\\
	0 & \Ext^2(A,B) & 0 & 0\\
	0 & \Ext^1(A,B) & \Ext^2(A,A)\oplus \Ext^2(B,B) & 0\\
	0 & 0 & \Ext^1(A,A)\oplus \Ext^1(B,B) & \Ext^2(B,A) \\
	0 & 0 & \Hom(A,A)\oplus \Hom(B,B) & \Ext^1(B,A) \\ \hline
	0 & 0 &  0 & \Hom(B,A)
	\end{array}$$
Note that the differential
\[d\colon E_1^{0,0}=\Hom(A,A)\oplus \Hom(B,B)\to E_1^{1,0}=\Ext^1(B,A)\]
maps $(\identity_A,0)$ and $(0,\identity_B)$ to the element $[E]\in \Ext^1(B,A)$ corresponds to the extension \eqref{des-seq}, we see $\dim \ker(d)\geq 1$. Thus from $\hom(E,E)=1$, we have $\dim E^{0,0}_2=\dim E^{0,0}_{\infty}=1$, hence $E_1^{1,-1}=\Hom(B,A)=0$. Moreover, since in each case we have $\Ext^2(E,E)=0$, we see  $$E^{-1,3}_1=E^{-1,3}_{\infty}=\Ext^2(A,B)=0,$$ which implies $\chi(A,B)=-\ext^1(A,B)\leq 0$. Then the first page becomes
 $$ E^{p,q}_1= \begin{array}{cc|cc}
	0 & 0 & 0 & 0\\
	0 & 0 & 0 & 0\\
	0 & \Ext^1(A,B) & \Ext^2(A,A)\oplus \Ext^2(B,B) & 0\\
	0 & 0 & \Ext^1(A,A)\oplus \Ext^1(B,B) & \Ext^2(B,A) \\
	0 & 0 & \Hom(A,A)\oplus \Hom(B,B) & \Ext^1(B,A) \\ \hline
	0 & 0 &  0 & 0
	\end{array}$$

\textbf{Case 1.} First, we assume that $\RHom(E,E)=\CC\oplus \CC^2[-1]$. When $X$ is a cubic threefold, we can assume that $[E]=v, v-w$ or $2v-w$ up to sign. In each case, a calculation using the Euler form and $\chi(A, B)\leq 0$ implies that $$[A]=[S^{-1}_{\Ku(X)}(E)],~\text{and}~[B]=[S_{\Ku(X)}(E)].$$ Thus $\chi(A,B)=-\ext^1(A,B)=0$. Therefore, we also have  $$E_1^{0,2}=E^{0,2}_{\infty}=\Ext^2(A,A)\oplus \Ext^2(B,B)=0.$$

Since $[A]$ is primitive, we see that $A$ has only one Jordan--H\"older factor with respect to $\sigma$ and is $\sigma$-stable. Hence $\hom(A,A)=1$. Now from $\chi(A,A)=-1$ and $\Ext^{\geq 2}(A,A)=0$, we see $\ext^1(A,A)=2$. By Lemma \ref{small-ext-stable}, $A$ is stable with respect to any Serre-invariant stability condition. Using $\Hom(A, E)=\Hom(E, S_{\Ku(X)}(A))\neq 0$ and the fact that $[E]=[S_{\Ku(X)}(A)]$ and $E, S_{\Ku(X)}(A)$ are stable with respect to any Serre-invariant stability condition, we get $A\cong S^{-1}_{\Ku(X)}(E)$. But this makes a contradiction, since $A, E \in \cA$ and
\[\Hom(A,E[-1])=\Hom(S^{-1}_{\Ku(X)}(E), E[-1])=\Ext^1(E,E)\neq 0.\]


When $X$ is a quartic double solid or a Gushel--Mukai threefold, we can assume that $[E]=v$ or $w$. Then a simple computation using $\chi(A,B)\leq 0$ shows that $[A]=0$ or $[B]=0$, which implies $A=0$ or $B=0$ and makes a contradiction.

\textbf{Case 2.} Now we assume that $\RHom(E,E)=\CC\oplus \CC^3[-1]$. In this case, $X$ is a quartic double solid or a Gushel--Mukai threefold and we can assume that $[E]=v-w$ or $v+w$. Then a simple computation using $\chi(A,B)\leq 0$ shows that $\{[A],[B]\}=\{v, w\}$ when $[E]=v+w$, and $\{[A],[B]\}=\{v,-w\}$ when $[E]=v-w$. Hence $\chi(A,B)=\ext^1(A,B)=0$. Therefore, we also have $$E_1^{0,2}=E^{0,2}_{\infty}=\Ext^2(A,A)\oplus \Ext^2(B,B)=0.$$ Since $[A]$ is primitive, we know that $A$ has only one Jordan--H\"older factor, hence is $\sigma$-stable. Thus $\hom(A,A)=1$, and $\ext^1(A,A)=2$ since $\chi(A,A)=-1$. By the previous case, $\tau(A):=S_{\Ku(X)}(A)[-2]$ is $\sigma$-stable as well. Moreover, by Lemma \ref{small-ext-stable} and Lemma \ref{bound-phase}, $\tau(A)$ is stable with respect to any Serre-invariant stability condition with the same phase as $A$. Thus $A,\tau(A)\in \cA$ by Lemma \ref{in-same-heart}. Hence $\phi_{\sigma}(A)=\phi_{\sigma}(\tau(A))$.

If $\Ext^2(B,A)=0$, then by Lemma \ref{mukai lemma} and (5) of Proposition \ref{S_prop} we get a contradiction. Thus we can assume that $\Ext^2(B,A)=\Hom(A, \tau(B))=\Hom(\tau(A), B)\neq 0$. By $\sigma$-stability of $\tau(A)$, we get $\phi_{\sigma}(\tau(A))\leq \phi^+_{\sigma}(B)$. Then from $\phi_{\sigma}(A)\geq \phi^+_{\sigma}(B)$ and $\phi_{\sigma}(A)=\phi_{\sigma}(\tau(A))$, we see $\phi_{\sigma}(A)=\phi_{\sigma}(\tau(A))=\phi^+_{\sigma}(B)$, hence from $\Hom(\tau(A), B)\neq 0$ and the $\sigma$-stability of $\tau(A)$, we have an injection $\tau(A)\hookrightarrow B$ in $\cA$. Since $[B]$ is primitive, by looking at the Jordan--H\"older filtration of the first Harder--Narasimhan factor of $B$ with respect to $\sigma$, we have an exact sequence $A'\to B\to B'$ in $\cA$ such that all Jordan--H\"older factors of $A'$ are isomorphic to $\tau(A)$ and $\Hom(A', B')=0$ as in Lemma \ref{des-triangle}. Applying the spectral sequence in Lemma \ref{spec-seq} as above, since $\Ext^{\geq 2}(B,B)=0$, we have $\Ext^2(A', B')=0$, which implies $\chi(A', B')\leq 0$. However, we know that $[A']=n[\tau(A)]=n[A]$ for an integer $n\geq 1$, hence
\[\chi(A', B')=\chi(n[A], [B]-n[A])=-n^2\chi([A], [A])=n^2>0\]
and we get a contradiction.
\end{proof}

As a corollary, we have:

\begin{corollary}\label{cor-image-Z}
Let $\sigma=(\cA, Z)$ be a stability condition on $\Ku(X)$ with homological dimension at most $2$. Then the image of the central charge $Z$ is not contained in a line, and we can find a Serre-invariant stability condition $\sigma'=(\cA', Z')$ on $\Ku(X)$ with $Z=Z'$. 
\end{corollary}

\begin{proof}
If $X$ is a cubic threefold or quartic double solid, let $l \subset X$ be a line. Then as $\ext^1(\cI_l,\cI_l) = 2$, by Theorem \ref{ext2_stable}, $\cI_l, S^{-1}_{\Ku(X)}(\cI_l)$ and $S_{\Ku}(\cI_l)$ are $\sigma$-stable.  Then the result follows from \cite[Remark 4.8]{pertusi2020some}.

When $X$ is a Gushel--Mukai threefold, the result follows from Theorem \ref{ext2_stable} and \cite[Lemma 4.7]{Pertusi2021serreinv}.
\end{proof}

We need two useful lemmas before proving our main theorem.

\begin{lemma}
Let $X$ be a cubic threefold, $\sigma=(\cA,Z)$ be a stability condition on $\Ku(X)$ with $\homdim(\sigma)\leq 2$ and $E\in \Ku(X)$ be an object with $\ext^1(E,E)=2$. Then we have
    \begin{equation}\label{eq-serre-action}
\phi_{\sigma}(E)+1\leq \phi_{\sigma}(S_{\Ku(X)}(E))<\phi_{\sigma}(E)+2.
\end{equation}
\end{lemma}

\begin{proof}
Up to shift, we can assume that $E\in \cA$. By Theorem \ref{ext2_stable}, $E, S_{\Ku(X)}(E)$ and $S^{-1}_{\Ku(X)}(E)$ are all $\sigma$-stable. Since $\Hom(E,E[1])=\Hom(E[1], S_{\Ku(X)}(E))\neq 0$, we have $$\phi_{\sigma}(E)+1\leq \phi_{\sigma}(S_{\Ku(X)}(E)).$$ And from Corollary \ref{cor-image-Z}, we can take a Serre-invariant stability condition $\sigma'=(\cA',Z')$ with $Z=Z'$ and $\phi_{\sigma}(E)=\phi_{\sigma'}(E)$. Hence $E\in \cA\cap \cA'$. Assume that $S_{\Ku(X)}(E)[m]\in \cA'$. To prove the statement, by Lemma \ref{bound-phase} and $Z=Z'$, we only need to show $S_{\Ku(X)}(E)[m]\in \cA$. To this end, assume that $S_{\Ku(X)}(E)[n]\in \cA$, then since $Z=Z'$, we see $n-m\in 2\ZZ$. From $\Hom(E, S_{\Ku(X)}(E))=\Hom(E, S_{\Ku(X)}(E)[n][-n])\neq 0$, we see $0\leq -n\leq 2$. And from $\Hom(E[1], S_{\Ku(X)}(E))=\Hom(E, S_{\Ku(X)}(E)[n][-1-n])\neq 0$, we see $0\leq -n-1\leq 2$. Therefore, we obtain $n=-1$ or $-2$. Since $n-m\in 2\ZZ$, we have $n=m$ and the result follows.
\end{proof}

\begin{lemma}\label{lem-D1D2}
Let $X$ be a cubic threefold, quartic double solid or Gushel--Mukai threefold. Then there exist two objects $D_1, D_2\in \Ku(X)$ with  $[D_1]=v$ and $[D_2]=w$, such that for any stability condition $\sigma$ on $\Ku(X)$ with $\homdim(\sigma)\leq 2$, $D_1$ and $D_2$ are $\sigma$-stable with phases
\begin{equation}\label{eq-lemD1D2}
    \phi_{\sigma}(D_1)-1<\phi_{\sigma}(D_2)<\phi_{\sigma}(D_1).
\end{equation}
\end{lemma}

\begin{proof}
    When $X$ is a cubic threefold or a quartic double solid, we set $D_1=\cI_l$ and $D_2=\mathbf{O}(\cI_l)[-1]$, where $l\subset Y$ is a line in $Y$ and $\mathbf{O}$ is the rotation functor defined in Section \ref{section-Ku}.  Therefore, by \cite[Lemma 5.16]{pertusi2020some},  we see
\[\RHom(D_i,D_i)=\CC\oplus \CC^2[-1],~\forall i\in \{1,2\}.\]
Hence $[D_1]=v$ and $[D_2]=w$, and are both $\sigma$-stable from Theorem \ref{ext2_stable}. Thus \eqref{eq-lemD1D2} follows from \cite[Remark 4.8]{pertusi2020some}, in which $\cJ_l:=\mathbf{O}^{-1}(\cI_l)[1]$. 

Now we assume that $X$ is a Gushel--Mukai threefold. We set $D_1=\cI_C$ and $D_2=F$, where $C\subset X$ is a general smooth conic and $F$ is a general rank $2$ non-locally free slope-stable sheaf on $X$ with $c_1(F)=-1, c_2(F)=5$ and $c_3(F)=0$. By the smoothness of $C$, we get $\RHom(D_1,D_1)=\CC\oplus \CC^2[-1]$. Since $C$ is general in the Fano surface of conics, by \cite[Proposition 7.1]{JLLZ} we have $D_1=\cI_C\in \Ku(X)$. Hence from Theorem \ref{ext2_stable}, $D_1$ is $\sigma$-stable. And by virtue of \cite[Theorem 8.1]{JLLZ}, $F$ fits into an exact sequence $0\to F\to \cE_X\to \oh_l(-1)\to 0$, where $l\subset X$ is a line. By \cite[Proposition 8.2]{JLLZ}, we see $D_2=F\in \Ku(X)$. Moreover, we can take $l$ to be general in the Hilbert scheme of lines on $X$ so that $\RHom(D_2,D_2)=\CC\oplus \CC^2[-1]$ (cf.~\cite[Section 8]{JLLZ}). Hence $D_2$ is $\sigma$-stable as well.

It remains to verify \eqref{eq-lemD1D2}. By \cite[Lemma 6.2]{JLLZ}, $\Hom(\cE_X, D_1)\neq 0$. Hence from the exact sequence above defining $F$, we see $\Hom(D_2,D_1)=\Hom(S^{-1}_{\Ku(X)}(D_1), D_2)\neq 0$. Thus we get $\phi_{\sigma}(S^{-1}_{\Ku(X)}(D_1))<\phi_{\sigma}(D_2)<\phi_{\sigma}(D_1)$. Note that $S^{-1}_{\Ku(X)}(D_1)$ is $\sigma$-stable as well, and by Lemma \ref{small-ext-stable} and Lemma \ref{bound-phase} $D_1$ and $S^{-1}_{\Ku(X)}(D_1)[2]$ are stable with respect to any Serre-invariant stability conditions with the same phase. Thus from $[D_1]=[S^{-1}_{\Ku(X)}(D_1)[2]]$ and Lemma \ref{in-same-heart}, we see $\phi_{\sigma}(D_1)=\phi_{\sigma}(S^{-1}_{\Ku(X)}(D_1)[2])$. Then we obtain $\phi_{\sigma}(D_1)-2<\phi_{\sigma}(D_2)<\phi_{\sigma}(D_1)$.

Now by Corollary \ref{cor-image-Z}, we can find a Serre-invariant stability condition $\sigma'=(\cA', Z')$ such that $Z=Z'$ and $\phi_{\sigma}(D_1)=\phi_{\sigma'}(D_1)$. Then we get \[\phi_{\sigma'}(D_1)-2<\phi_{\sigma}(D_2)<\phi_{\sigma'}(D_1).\] However, from $Z=Z'$, we see $\phi_{\sigma}(D_2)-\phi_{\sigma'}(D_2)\in 2\ZZ$, which forces $\phi_{\sigma}(D_2)=\phi_{\sigma'}(D_2)$. Then it remains to check $\phi_{\sigma'}(D_1)-1<\phi_{\sigma'}(D_2)<\phi_{\sigma'}(D_1)$. By Theorem \ref{unqiue}, up to $\GL$-action, we can take $\sigma'$ to be the one constructed in \cite{bayer2017stability}, then the result follows from a direct computation of tilt-stability.
\end{proof}

Now we are ready to prove our main theorems. We first prove that a stability condition $\sigma$ on $\Ku(X)$ with $\homdim(\sigma)\leq 2$ is determined by its central charge.

\begin{theorem}\label{cri_thm_1}
Let $X$ be a cubic threefold, quartic double solid or Gushel--Mukai threefold, and $\sigma_1=(\cA_1,Z_1),\sigma_2=(\cA_2,Z_2)$ be a pair of stability conditions on $\Ku(X)$. If $Z_1=Z_2$ and $\homdim(\sigma_i)\leq 2$ for any $i\in \{1,2\}$, then 
\[\sigma_1=[2m]\cdot\sigma_2\]
for an integer $m\in \ZZ$.
\end{theorem}

\begin{proof}
We take two objects $D_1$ and $D_2$ as in Lemma \ref{lem-D1D2}. Note that $[D_1]=v$ and $[D_2]=w$ and are both $\sigma_1$-stable and $\sigma_2$-stable. Moreover, we have
\begin{equation}\label{Di-phase}
    \phi_{\sigma_k}(D_1)-1<\phi_{\sigma_k}(D_2)<\phi_{\sigma_k}(D_1)
\end{equation}
for any $k\in \{1,2\}.$ Up to shift, we can furthermore assume that $D_1\in \cA_1\cap \cA_2$. Thus from \eqref{Di-phase}, either $D_2\in \cA_1\cap \cA_2$ or $D_2[1]\in \cA_1\cap \cA_2$. Moreover, since $Z_1=Z_2$, we get $\phi_{\sigma_1}(D_i)=\phi_{\sigma_2}(D_i)$ for $i\in \{1,2\}$.

\textbf{Claim 1.} \emph{If a non-zero object $E\in \cA_i$ satisfies $E[n]\in \cA_j$, then $n=0$, where $\{i, j\}=\{1,2\}$.} Note that from $Z_1=Z_2$, we have $n\in 2\ZZ$.

Since the claim and assumptions are symmetry between $\sigma_1$ and $\sigma_2$, in the following we assume that $i=1$ and $j=2$. Then we have 
\begin{equation}\label{eq-n}
    n<\phi^-_{\sigma_2}(E), \quad \phi^+_{\sigma_2}(E)\leq n+1.
\end{equation}

Let $[E]=av+bw$ for $a,b\in \ZZ$.  First, we assume that $X$ is a cubic threefold. Assume $F\in \cA_1\cap \cA_2$ be an object with $\ext^1(F,F)=2$. Then $F, S_{\Ku(X)}(F)$ and $S^{-1}_{\Ku(X)}(F)$ are $\sigma_1$-stable and $\sigma_2$-stable. Note that if $\chi(F,E)<0$ or $\chi(E,F)<0$, then from $\homdim(\sigma_1)\leq 2$, we get $\Hom(F,E[1])=\Hom(E, S_{\Ku(X)}(F)[-1])\neq 0$ or $\Hom(E,F[1])=\Hom(S^{-1}_{\Ku(X)}(F)[1],E)\neq 0$, which implies 
\[\phi_{\sigma_2}(F)-1\leq \phi^+_{\sigma_2}(E),\quad \phi_{\sigma_2}^-(E)\leq\phi_{\sigma_2}(S_{\Ku(X)}(F))-1\]
or
\[\phi_{\sigma_2}(S^{-1}_{\Ku(X)}(F))+1\leq \phi^+_{\sigma_2}(E),\quad \phi_{\sigma_2}^-(E)\leq \phi_{\sigma_2}(F)+1.\]
But by \eqref{eq-n} and \eqref{eq-serre-action}, we always have $-1<n+1$ and $n<2$, which implies $n=0$ since $n\in 2\ZZ$. Therefore, in the following, we only need to find an object $F\in \cA_1\cap \cA_2$ with $\ext^1(F, F)=2$ and $\chi(F, E)<0$ or $\chi(E, F)<0$.

\begin{itemize}
    \item If $b=0$, then $[E]=a[D_1]$ and $a>0$. In this case, we have $$\chi(D_1, E)=\chi(E, D_1)=-a<0.$$ 

    \item Assume  $D_2\in \cA_1\cap \cA_2$. If $a=0$, then $b>0$. In this case, we have $$\chi(E, D_2)=-a-b<0.$$ 

     \item Assume $D_2\in \cA_1\cap \cA_2$. If $a\neq 0$, $b\neq 0$, then from
     \[\Im(Z_1(E))=a\cdot\Im(Z_1(D_1))+b\cdot\Im(Z_1(D_2))\geq 0,\]
     either $a<0<b$ or $0<a$. Then we have either $\chi(D_2, E)=-b<0$ or $\chi(E, D_1)=-a<0$. 

   \item Assume $D_2[1]\in \cA_1\cap \cA_2$. If $a=0$, then $b<0$. In this case, we have $$\chi(E, D_2[1])=a+b<0.$$ 

     \item Assume $D_2[1]\in \cA_1\cap \cA_2$. If $a\neq 0$, $b\neq 0$, then from
     \[\Im(Z_1(E))=a\cdot\Im(Z_1(D_1))-b\cdot\Im(Z_1(D_2[1]))\geq 0,\]
     either $0<a$ or $b<0, a<0$. Then we have either $\chi(E, D_1)=-a<0$ or $\chi(E, D_2[1])=a+b<0$. 
\end{itemize}

Now we assume that $X$ is a quartic double solid or a Gushel--Mukai threefold.

\begin{itemize}
    \item If $b=0$, then $a>0$. In this case, we have $$\chi(D_1, E)=\chi(E, D_1)=-a<0,$$ hence from $D_1, E\in \cA_1$ we get $\Hom(D_1, E[1])\neq 0$ and $\Hom(E, D_1[1])\neq 0$. Then we have \[\phi_{\sigma_2}(D_1)-1\leq \phi^+_{\sigma_2}(E),\quad \phi_{\sigma_2}^-(E)\leq \phi_{\sigma_2}(D_1)+1.\]
    But by \eqref{eq-n}, we get $-1<n+1$ and $n<2$, which implies $n=0$ since $n\in 2\ZZ$.

    \item Assume $D_2\in \cA_1\cap \cA_2$. If $b\neq 0$, then $b>0$. In this case, we have $$\chi(D_2, E)=\chi(E, D_2)=-b<0.$$ Then by the same argument above, we get $\Hom(D_2, E[1])\neq 0$ and $\Hom(E, D_2[1])\neq 0$. Then we have \[\phi_{\sigma_2}(D_2)-1\leq \phi^+_{\sigma_2}(E),\quad \phi_{\sigma_2}^-(E)\leq \phi_{\sigma_2}(D_2)+1.\]
    But by \eqref{eq-n}, we get $-1<n+1$ and $n<2$, which implies $n=0$ since $n\in 2\ZZ$.

    \item Assume $D_2[1]\in \cA_1\cap \cA_2$. If $b\neq 0$, then $b<0$. In this case, we have $$\chi(D_2[1], E)=\chi(E, D_2[1])=b<0.$$ Then by the same argument above, we get $\Hom(D_2[1], E[1])\neq 0$ and $\Hom(E, D_2[2])\neq 0$. Then we have \[\phi_{\sigma_2}(D_2[1])-1\leq \phi^+_{\sigma_2}(E),\quad \phi_{\sigma_2}^-(E)\leq \phi_{\sigma_2}(D_2[1])+1.\]
    But by \eqref{eq-n}, we get $-1<n+1$ and $n<2$, which implies $n=0$ since $n\in 2\ZZ$.
\end{itemize}

\textbf{Claim 2.} \emph{A non-zero object $E\in \Ku(X)$ is in $\cA_1$ if and only if $E$ is in $\cA_2$.}

We prove this claim by induction on $\ext^1(E,E)$. When $\ext^1(E,E)\leq 2$, we have $\ext^1(E,E)=2$ by Proposition \ref{S_prop} (5), hence $E$ is $\sigma_i$-stable by Theorem \ref{ext2_stable}. then the result follows from our assumption $\phi_{\sigma_1}(E)=\phi_{\sigma_2}(E)$. Now we assume that the claim holds for any object $E$ with $\ext^1(E, E)<N$ for an integer $N>2$. We are going to prove the claim for $\ext^1(E, E)=N$. To this end, we first assume that $E\in \cA_1$ with $\ext^1(E, E)=N$. If $E$ has at least two cohomology objects with respect to $\cA_2$, we denote $a,b\in \ZZ$ by unique integers satisfy
\[\cH^{-a}_{\cA_2}(E)\neq 0, \cH^{k}_{\cA_2}(E)=0, \forall k<-a\]
and
\[\cH^{-b}_{\cA_2}(E)\neq 0, \cH^{k}_{\cA_2}(E)=0, \forall k>-b.\]
Therefore, we have two non-zero maps $\cH^{-a}_{\cA_2}(E)[a]\to E$ and $E\to \cH^{-b}_{\cA_2}(E)[b]$. Using (5) of Proposition \ref{S_prop} and Lemma \ref{ext1-heart}, we see $$\ext^1(\cH^{-a}_{\cA_2}(E),\cH^{-a}_{\cA_2}(E))<N,~\text{and}~\ext^1(\cH^{-b}_{\cA_2}(E),\cH^{-b}_{\cA_2}(E))<N.$$ Thus by the induction hypothesis, we get $\cH^{-a}_{\cA_2}(E),\cH^{-b}_{\cA_2}(E)\in \cA_1$. But since $E\in \cA_1$, we have $a\leq 0$ and $b\geq 0$, which contradicts $a>b$. This implies $E$ is in $\cA_2$ up to a shift, but Claim 1 shows $E\in \cA_2$. When $E\in \cA_2$ with $\ext^1(E, E)=N$, the same argument as above also shows that $E\in \cA_1$. This completes our induction argument and proves Claim 2.

Therefore, by the Claim 2 above, we have $\cA_1=\cA_2$, which together with $Z_1=Z_2$ implies $\sigma_1=\sigma_2$.
\end{proof}

Since we know that the homological dimension of any Serre invariant stability condition on $\Ku(X)$ is at most $2$, we get the following criterion.

\begin{theorem}\label{cri_thm}
Let $X$ be a cubic threefold, quartic double solid or Gushel--Mukai threefold, and $\sigma$ be a stability condition on $\Ku(X)$. Then the following conditions are equivalent:

\begin{enumerate}
    \item $\homdim(\sigma)\leq 2$,

    \item $\gldim(\sigma)\leq 2$, and

    \item $\sigma$ is Serre-invariant.
\end{enumerate}

\end{theorem}

\begin{proof}
By Proposition \ref{S_prop}, we see (3) implies (2), and by Lemma \ref{global_homo} we have (2) implies (1). 

By Theorem \ref{ext2_stable} and Corollary \ref{cor-image-Z}, if $\homdim(\sigma)\leq 2$, then we can find a Serre-invariant stability condition $\sigma'$ such that $Z=Z'$. Then (1) implies (3) follows from Theorem \ref{cri_thm_1}.
\end{proof}

\subsection{The Serre-invariant component}

Let $\mathsf{K}\subset \Stab(\Ku(X))$ be the subspace of all Serre-invariant stability conditions on $\Ku(X)$. By Theorem \ref{unqiue}, we have $\mathsf{K}=\GL$. It is clear that $\mathsf{K}$ is contractible since $\GL$ is.

\begin{corollary}\label{cor-component}
Let $X$ be a cubic threefold, quartic double solid or Gushel--Mukai threefold. Then $\mathsf{K}$ is a contractible connected component of $\Stab(\Ku(X))$. 
\end{corollary}

\begin{proof}
By \cite[Remark 3.10]{pertusi2020some},  $\mathsf{K}$ is an open subset of a connected component of $\Stab(\Ku(X))$. Since $\gldim\colon \Stab(\Ku(X))\to \mathbb{R}$ is continuous by \cite[Lemma 5.7]{IQ:q-stability}, the subspace $\gldim^{-1}([0,2])\subset \Stab(\Ku(X))$ is closed. Moreover, by Theorem \ref{cri_thm}, $\mathsf{K}=\gldim^{-1}([0,2])$. Since $\mathsf{K}=\GL \cong \CC\times \mathbb{H}$ is connected, open and closed, it is a contractible connected component.
\end{proof}

For an arbitrary stability condition $\sigma$ on $\Ku(X)$, Proposition \ref{S_prop} (5) implies $\homdim(\sigma)\geq 1$. So it is natural to make the following conjecture, which together with Theorem \ref{cri_thm} could describe all stability conditions on $\Ku(X)$:

\begin{conjecture}\label{conj}
Let $X$ be a cubic threefold, quartic double solid or Gushel--Mukai threefold. Then for any stability condition $\sigma$ on $\Ku(X)$, we have $\homdim(\sigma)\leq 2.$
\end{conjecture}

\bibliographystyle{plain}
{\small{\bibliography{stab}}}

\end{document}